\documentclass[pdflatex,sn-mathphys-num]{sn-jnl}

\usepackage{graphicx}%
\usepackage{subcaption}
\usepackage{multirow}%
\usepackage{amsmath,amssymb,amsfonts,eqnarray,cite}%
\usepackage{amsthm}%
\usepackage{mathrsfs}%
\usepackage[title]{appendix}%
\usepackage{xcolor}%
\usepackage{textcomp}%
\usepackage{manyfoot}%

\theoremstyle{thmstyleone}%
\newtheorem{theorem}{Theorem}

\newtheorem{corollary}[theorem]{Corollary}

\theoremstyle{thmstyletwo}%
\newtheorem{remark}{Remark}%

\theoremstyle{thmstylethree}%
\newtheorem{definition}{Definition}%

\begin{document}

\title{Certain families of series involving Hurwitz and Lerch zeta functions}

\author[1]{\fnm{Dora} \sur{Pokaz}}\email{dora.pokaz@grad.unizg.hr}

\author[2]{\fnm{Mihaela Ribi\v{c}i\'{c}} \sur{Penava}}\email{mihaela@mathos.hr}
\author[3]{\fnm{Yilmaz} \sur{Simsek}}\email{ysimsek@akdeniz.edu.tr}


\affil[1]{\orgdiv{Faculty of Civil Engineering}, \orgname{University of Zagreb}, \country{Zagreb, Croatia}}
\affil[2]{\orgdiv{School of Applied Mathematics and Informatics}, \orgname{Josip Juraj Strossmayer University of Osijek}, \country{Osijek, Croatia}}
\affil[3]{\orgdiv{Department of Mathematics}, \orgname{Faculty of Science, University of Akdeniz, TR-07058}, \country{Antalya, Turkey}}


\abstract{This paper deals with a certain family of series which is derived from generating functions of the Frobenius-Euler and Apostol type polynomials and numbers.  When special cases of this series are examined, it is shown that it generates some well-known families of zeta functions, such as the Hurwitz zeta and the Lerch zeta, etc.  We also give an open problem for these relation. We give  analytic continuation of this function. By using the Cauchy Residue
	Theorem,  We show that this series interpolates the Frobenius-Euler type polynomials at negative integers. Moreover, by
applying the Laplace transform to the generating function for these
polynomials, we get infinite series representations involving the Lerch zeta 
function and these polynomials. In addition, We show that this interpolation function can be representation in terms of the Lerch zeta
function and the Hurwitz zeta function. Some special values for the results are also given.}




\maketitle

\section{Introduction}

The history of series, and especially the zeta function, dates back a long time. Perhaps its most magnificent development began with the summation of power series in the 18th century and later reached its peak with the construction of prime number theory in the 19th century. In my opinion, it is well known that the Swiss mathematician Jakob Bernoulli (1654-1705), who dedicated his life to this field, invented the numbers bearing his surname, namely Bernoulli numbers, and that his famous student and fellow Swiss mathematician Leonhard Euler related them to the zeta function. This function achieved its present-day fame thanks to the famous German mathematician Bernhard Riemann, who laid its foundations in the complex plane. Before Riemann, such functions were only worked with in the real part, but Riemann extended both the real and imaginary parts, defining this zeta function for complex numbers. Today, this function is known as the Riemann zeta function. The fundamental spirit of this function is that, thanks to its analytic continuity, it offers vital mathematical applications and uses in other fields as well. Revolutionary developments in number theory have occurred thanks to this function. Another important area is its relationship with the Riemann hypothesis. It is also worth highlighting the impact of the zeta function in computer science, particularly in areas such as cryptography and the analysis of random matrices, which are among the most important topics in number theory. Beyond mathematics, we see its applications in fundamental areas of physics and engineering, including quantum mechanics  (\textit{cf}. 
\cite{Barnes,knopp,Knuf,SrivatavaChoi,Tichmarsh}). The Riemann zeta
function, whose definition is given by%
\begin{equation*}
	\zeta (w)=\sum\limits_{v=1}^{\infty }\frac{1}{v^{w}},
\end{equation*}%
where $w\in \mathbb{C}$, set of complex numbers, with $Re(w)>1$ (\textit{cf}. \cite{Barnes,SrivatavaChoi,Tichmarsh}). In this paper, we assume that $\ln w$ denotes the principal branch of the many-valued function $\ln w$ with the imaginary part  $ Im(w))$ constrained by
$\pi <Im(w)\leq \pi$.

The other well-known zeta function is the following Hurwitz zeta function, which is known as a generalization of the Riemann zeta
function:
\begin{equation*}
	\zeta (w,u)=\sum\limits_{v=1}^{\infty }\frac{1}{(u+v)^{w}},
\end{equation*}%
where $u\notin \mathbb{Z}_{0}^{-}=\left\{ \ldots ,-2,-1,0\right\} $ and $%
w\in \mathbb{C}$ with $Re(w)>1$ (\textit{cf}. \cite{Barnes,Batemen,Bayad,SrivatavaChoi}).

It is clear that $$\zeta (w):=\zeta (w,1)$$ (\textit{cf.} \cite{Batemen,SrivatavaChoi}). There are various generalizations of the Riemann zeta
function and the Hurwitz zeta function. One of these is the following: the Lerch transcendent (the Lerch zeta function) $\Phi (w;z,x)$:
\begin{equation}
	\Phi (w;z,x)=\sum_{l=0}^{\infty }\frac{z^{l}}{\left( l+x\right) ^{w}},
	\label{LZ}
\end{equation}%
where $x\notin \mathbb{Z}_{0}^{-}$; $w\in \mathbb{C}$ when $\left\vert
z\right\vert <1$; and $Re(w)>1$ when $\left\vert z\right\vert =1$ (\textit{cf%
}. \cite{Barnes,Batemen,Bayad,TkimJKMS,SrivatavaChoi}). The function $\Phi (w;z,x)$ was first
studied by Lipschitz (1832--1903) and later by Lerch (1860--1922). This
function contains the following special functions:

The Riemann function:%
\begin{equation*}
	\zeta (w):=\Phi (w;1,1)
\end{equation*}%
(\textit{cf}. \cite{SrivatavaChoi}). The Hurwitz (or generalized) zeta
functions:%
\begin{equation*}
	\zeta (w,x):=\Phi (w;1,x)
\end{equation*}%
(\textit{cf}. \cite{SrivatavaChoi}). The Polylogarithm function can also be
derived from (\ref{LZ}), that is,
\begin{equation*}
	\mathrm{Li}_{w}(z):=z\Phi (w;z,1)
\end{equation*}%
(\textit{cf}. \cite{Barnes,Batemen,Bayad,SrivatavaChoi}).

It is well-known that there are various types of families of series involving zeta functions. Special cases of these types of zeta functions are related to the Riemann zeta function.
These types of functions have important applications in many different areas,
especially in analytic number theory and mathematical physics (\textit{cf}. \cite{Barnes,Batemen,Bayad,SimsekMTJPAM2021,SrivatavaChoi}).

The main motivation of this paper is to not only define a certain family of series involving the Hurwitz zeta function, the Lerch type zeta functions, but also give some fundamental properties of this series. 

We define the following family of series involving the Hurwitz zeta functions and the Lerch zeta function:

\begin{definition}
	Let $s\in \mathbb{C}$ and $v\in \mathbb{N}$ with $v\geq 3$. We define $ 	
	Z(s;x,v)$ by%
	\begin{eqnarray}
	Z(s;x,v)&=&	-\frac{s(s+1)\cdots(s+v-1)}{(v-1)!}\sum\limits_{m=0}^{\infty }\sum\limits_{k_{v-2}=0}^{m}\cdots
		\sum\limits_{k_{3}=0}^{k_{4}}\sum\limits_{k_{2}=0}^{k_{3}}\notag\\&&\times\sum%
		\limits_{k_{1}=0}^{k_{2}}\frac{1}{
			1^{k_{1}}2^{k_{2}-k_{1}}3^{k_{3}-k_{2}}4^{k_{4}-k_{3}}%
			\cdots (v-1)^{m-k_{v-2}}\left( x-1+m\right) ^{s+v}}	,		
	 \label{DM-1}
	\end{eqnarray}
		where $x-1\notin \mathbb{Z}_{0}^{-}$ and $Re(s+v)>1$; for $v=2$, we have
	\begin{equation*}
		Z(s;x,2)=-s(s+1)\zeta (s+2,x-1),
			\end{equation*}
	where $x-1\notin \mathbb{Z}_{0}^{-}$ and $Re(s+2)>1$.
\end{definition}

We summarize the results of this paper as follows:

In Section \ref{pr}, some standard notations and definitions are given.

In Section \ref{S1}, by using the finite geometric sum formula, some special values of the series $Z(s;x,v)$ in terms of the Lerch zeta function $\Phi (s;z,x)$ and
the Hurwitz zeta function $\zeta (s,x)$ can be given. We give an open problem.

In Section \ref{S2}, We show that the function $Z(s;x,v)$ have analytic
continuations on the complex $s$-plane with the aid of integral representation of the function $Z(s;x,v)$. We also show that the function $Z(s;x,v)$ interpolation function for the polynomials $\ell _{n}(x;v)$ by using integral representation of the function $Z(s;x,v)$ and the Cauchy Residue Theorem at negative integers.

In Section \ref{S3}, applying the Laplace
transform to the generating function, we derive infinite series
representations for the Frobenius-Euler type Simsek polynomials and the Lerch zeta function. We also give some formulas involving the series $Z(s;x,v)$ in terms of the Lerch zeta function $\Phi (s;z,x)$ and
the Hurwitz zeta function $\zeta (s,x)$. 

\section{Preliminaries} \label{pr}

We start this section with the following standard notations and definitions in \cite{AgyuzFilomat}--\cite{Tichmarsh}:

Let $\mathbb{N}$, $\mathbb{Z}$, $\mathbb{Q}$, and $\mathbb{R}$ denote the
set of positive integers, the set of integers, the set of rational numbers,
and the set of real numbers, respectively. Let $\mathbb{N}_{0}=\mathbb{N}%
\cup \left\{ 0\right\} $.

The Bernoulli polynomials $B_{k}(y)$ are defined by%
\begin{equation}
	F_{B}(w,y)=\frac{w}{e^{w}-1}e^{wy}=\sum\limits_{k=0}^{\infty }\frac{B_{k}(y)%
	}{k!}w^{k},  \label{B.h}
\end{equation}%
where $\left\vert w\right\vert <2\pi $ (\textit{cf}. \cite{AgyuzFilomat,Apostol,simsekRJMP.mtjpam,SimsekBook,SrivatavaChoi}).

The Apostol-Bernoulli polynomials $\mathcal{B}_{k}(y;\beta )$ are defined by%
\begin{equation*}
	\frac{w}{\beta e^{w}-1}e^{wy}=\sum\limits_{k=0}^{\infty }\frac{\mathcal{B}%
		_{k}(y;\beta )}{k!}w^{k}
\end{equation*}%
(\textit{cf}. \cite{Apostol}).

The Stirling numbers of the first kind $S_{1}(v,d)$ are defined by%
\begin{equation}
	\left( \log (1+w)\right) ^{d}=\sum_{v=0}^{\infty }\frac{d!S_{1}(v,d)}{v!}%
	w^{v}.  \label{S1}
\end{equation}%
Using (\ref{S1}), we have 
\begin{equation*}
	S_{1}(v,d)=0
\end{equation*}%
if $d>v$ and%
\begin{equation}
	y_{(v)}=\sum\limits_{d=0}^{v}S_{1}(v,d)y^{d}  \label{S1fal}
\end{equation}%
(\textit{cf}. \cite{SimsekBook,SrivatavaChoi}).

The Stirling numbers of the second kind $S_{2}(v,d)$ are defined by%
\begin{equation}
	\left( e^{w}-1\right) ^{d}=\sum_{v=0}^{\infty }\frac{d!S_{2}(v,d)}{v!}w^{v},
	\label{SN-1}
\end{equation}%
where $d\in \mathbb{N}_{0}$ 
\begin{equation}
	y^{v}=\sum\limits_{d=0}^{v}S_{2}(v,d)y_{(d)}  \label{S2-1a}
\end{equation}%
(\textit{cf}. \cite{SrivatavaChoi}).

Simsek \cite{SimsekMMAS2021} defined the following numbers $\ell _{n}\left(
v\right) $:%
\begin{equation}
	\sum_{n=0}^{\infty }\ell _{n}\left( v\right) \frac{w^{n}}{n!}=\frac{w^{v}}{%
		\prod\limits_{j=0}^{v-1}\left( e^{w}-j\right) }  .\label{1g2}
\end{equation}%

By the following identities, we see that the numbers $\ell _{n}(v)$ are related to the many known special numbers:%
\begin{equation*}
	\ell _{n}(2)=\sum_{q=0}^{n}\ell _{q}(1)B_{n-q},
\end{equation*}%
\begin{equation*}
	\ell _{n}(2)=\sum_{q=0}^{n}(-1)^{q+1}qB_{n-q}=nB_{n-1}(-1),
\end{equation*}%
and also
\begin{eqnarray*}
	\ell _{n}(3) &=&-n\sum_{q=0}^{n-1}\sum_{k=0}^{n-q-1}\binom{n-1}{q}k!\ell
	_{q}(2)S_{2}\left( n-q-1,k\right)  \\
	&=&-n\sum_{q=0}^{n-1}\binom{n-1}{q}\ell _{q}(2)w_{g}(n-q-1) \\
	&=&\frac{1}{2}\sum_{q=0}^{n}\binom{n}{q}\mathcal{B}_{n-q}\left( \frac{1}{2}%
	\right) \ell _{q}(2),
\end{eqnarray*}%
where $w_{g}(n)$ denotes the Fubini numbers (\textit{cf}. \cite%
{SimsekMMAS2021}, see also \cite{KilarSMontes}).

The polynomial $\ell _{n}(x;v)$ is defined by%
\begin{equation}
	F(w,x;v)
=	\frac{w^{v}}{\prod\limits_{j=0}^{v-1}\left( e^{w}-j\right) }%
	e^{xw}=\sum_{n=0}^{\infty }\ell _{n}(x;v)\frac{w^{n}}{n!},  \label{1g2AA}
\end{equation}%
which implies%
\begin{equation}
	\ell _{n}(x;v)=\sum_{j=0}^{n}\binom{n}{j}x^{j}\ell _{n-j}(v)  \label{1g2AA1}
\end{equation}%
(\textit{cf}. \cite[Theorem 7]{SimsekMMAS2021}).

Simsek \cite{SimsekMMAS2021} showed that \cite{SimsekMMAS2021} special values of the polynomials $\ell _{n}(x;v)$ reduces to the Frobenius–Euler polynomials, the Apostol-Bernoulli and Euler polynomials and also the Bernoulli polynomials.

Here we note that the polynomials $\ell _{n}\left(x; v\right)$ are called the Frobenius–Euler type Simsek polynomials by Agyuz \cite{AgyuzFilomat}. After that, Rao et al \cite{Rao,Rao2} and Kilar \cite{KilarMontes} have studied applications of these polynomials. For other similar studies, see also the authors' paper  \cite{Peneva} and  Kucukoglu \cite{IK,KucukogluMTJPAM}. They gave many interesting applications of these polynomials. 

By using (\ref{1g2AA}), we have%
\begin{equation*}
	\frac{e^{w(x-1)}}{e^{w}-2}\frac{1}{e^{w}-3}\cdots \frac{1}{e^{w}-v+1}=\left(
	e^{w}-1\right) \sum_{n=0}^{\infty }\ell _{n}(x;v)\frac{w^{n}}{n!}.
\end{equation*}%
Therefore, we get%
\begin{equation}
	\frac{w^{v}}{\prod\limits_{j=2}^{v-1}\left( \frac{1}{j}e^{w}-1\right) }%
	e^{(x-1)w}=(v-1)!\left( e^{w}-1\right) \sum_{n=0}^{\infty }\ell _{n}(x;v)%
	\frac{w^{n}}{n!} \label{dp1}
\end{equation}
and%
\begin{eqnarray}
	\sum_{n=0}^{\infty }\ell _{n}(x;v)\frac{w^{n}}{n!} &=&\frac{w^{v}}{%
		\prod\limits_{j=1}^{v-1}\left( 1-je^{-w}\right) }e^{\left( x-v\right) w}
	\label{dp7} \\
	&=&w^{v}e^{\left( x-v\right) w}\prod\limits_{j=1}^{v-1}\sum_{n=0}^{\infty
	}j^{n}e^{-nw}.  \notag
\end{eqnarray}
Recently, the authors \cite{Peneva} gave many novel formulas for these polynomials $\ell _{n}(x;v)$.

More detailed information on the properties of special numbers and polynomials, and the functions to which they relate, can be found in the reference cited therein (\textit{cf}. 
\cite{AgyuzFilomat,Apostol,Barnes,Batemen,Bayad,Jordan,TkimJKMS,knopp,Knuf,KucukogluMTJPAM,MilneThomson,Rao,Rao2,Peneva,simsekRJMP.mtjpam,SimsekMMAS2021,SrivatavaChoi,Tichmarsh}).
\section{Special values of the series $Z(s;x,v)$} \label{S1}

In this section, we give some special values of the series
$Z(s;x,v)$ in terms of the Lerch zeta function $\Phi (s;z,x)$ and
the Hurwitz zeta function $\zeta (s,x)$.

We now evaluate the function $Z(s;x,v)$ at $v=3$. Putting $v=3$ in Eq. (%
\ref{DM-1}) 
\[
Z(s;x,3)=-\frac{s(s+1)(s+2)}{2}\sum\limits_{m=0}^{\infty
}\sum\limits_{k_{1}=0}^{m}\frac{1}{2^{m-k_{1}}\left( x-1+m\right) ^{s+3}}.
\]%
Thus 
\[
Z(s;x,3)=-\frac{s(s+1)(s+2)}{2}\sum\limits_{m=0}^{\infty }\frac{1}{%
	2^{m}\left( x-1+m\right) ^{s+3}}\sum\limits_{k_{1}=0}^{m}2^{k_{1}}.
\]%
Since%
\[
\sum\limits_{k_{1}=0}^{m}2^{k_{1}}=\frac{2^{m+1}-1}{2-1},
\]%
after some calculations, we have%
\[
Z(s;x,3)=-\frac{s(s+1)(s+2)}{2}\sum\limits_{m=0}^{\infty }\frac{2^{m+1}-1%
}{2^{m}\left( x-1+m\right) ^{s+3}}.
\]%
Combining the above equation with definitions of the Hurwitz zeta function and the 
Lerch zeta function, we get%
\[
Z(s;x,3)=-s(s+1)(s+2)\left(\zeta (s+3,x-1)-\frac{1}{2}\Phi \left( s+3;%
\frac{1}{2},x-1\right)\right) ,
\]%
where $x-1\notin \mathbb{Z}_{0}^{-}$ and $Re(s+3)>1$.

We now evaluate the function $Z(s;x,4)$. Putting $v=4$ in (\ref{DM-1}),
we obtain%
\begin{eqnarray*}
	Z(s;x,4)=	-\frac{s(s+1)(s+2)(s+3)}{6}\sum\limits_{m=0}^{\infty }\sum\limits_{k_{2}=0}^{m}\sum	\limits_{k_{1}=0}^{k_{2}}\frac{1}{
		1^{k_{1}}2^{k_{2}-k_{1}}3^{m-k_{2}}\left( x-1+m\right) ^{s+4}}
\end{eqnarray*}
Since%
\begin{eqnarray*}
	\sum\limits_{k_{2}=0}^{m}\sum\limits_{k_{1}=0}^{k_{2}}\frac{1}{%
		2^{k_{2}-k_{1}}3^{m-k_{2}}} &=&\frac{1}{3^{m}}\sum\limits_{k_{2}=0}^{m}%
	\left( \frac{3}{2}\right) ^{k_{2}}\sum\limits_{k_{1}=0}^{k_{2}}2^{k_{1}} \\
	&=&\frac{1}{3^{m}}\sum\limits_{k_{2}=0}^{m}\left( \frac{3}{2}\right) ^{k_{2}}%
	\frac{2^{k_{2}+1}-1}{2-1} \\
	&=&3+\frac{1}{3^{m}}-\frac{3}{2^{m}},
\end{eqnarray*}%
we get%
\begin{eqnarray*}
	Z(s;x,4) &=&-\frac{s(s+1)(s+2)(s+3)}{2}\sum\limits_{m=0}^{\infty }\frac{1%
	}{\left( x-1+m\right) ^{s+4}} \\
	&&-\frac{s(s+1)(s+2)(s+3)}{6}\sum\limits_{m=0}^{\infty }\frac{1}{3^{m}\left(
		x-1+m\right) ^{s+4}} \\
	&&+\frac{s(s+1)(s+2)(s+3)}{2}\sum\limits_{m=0}^{\infty }\frac{1}{2^{m}\left(
		x-1+m\right) ^{s+4}}.
\end{eqnarray*}%
Combining the above equation with the definitions of the Hurwitz zeta function and the 
Lerch zeta function, we get%
\begin{eqnarray*}
	Z(s;x,4) &=&-\frac{s(s+1)(s+2)(s+3)}{2}\zeta (s+4,x-1) \\
	&&-\frac{s(s+1)(s+2)(s+3)}{6}\Phi \left( s+4;\frac{1}{3},x-1\right)  \\
	&&+\frac{s(s+1)(s+2)(s+3)}{2}\Phi \left( s+4;\frac{1}{2},x-1\right) ,
\end{eqnarray*}%
where $x-1\notin \mathbb{Z}_{0}^{-}$ and $Re(s+4)>1$.

We now evaluate the function $Z(s;x,5)$. Putting $v=5$ in (\ref{DM-1}),
we obtain 
\[
Z(s;x,5)=-\frac{s(s+1)(s+2)(s+3)(s+4)}{24}\sum\limits_{m=0}^{\infty
}\sum\limits_{k_{3}=0}^{m}\sum\limits_{k_{2}=0}^{k_{3}}\sum%
\limits_{k_{1}=0}^{k_{2}}\frac{1}{1^{k_{1}}2^{k_{2}-k_{1}}3^{k_{3}-k_{2}}4^{m-k_{3}}%
	\left( x-1+m\right) ^{s+5}}.
\]%
Since
\begin{align*}
	\sum_{k_{3}=0}^{m}
	\sum_{k_{2}=0}^{k_{3}}
	\sum_{k_{1}=0}^{k_{2}}
	\frac{1}{
		2^{k_{2}-k_{1}}
		3^{k_{3}-k_{2}}
		4^{m-k_{3}}
	}
	&=
	\frac{1}{4^{m}}
	\sum_{k_{3}=0}^{m}
	\left(\frac{4}{3}\right)^{k_{3}}
	\sum_{k_{2}=0}^{k_{3}}
	\left(\frac{3}{2}\right)^{k_{2}}
	\sum_{k_{1}=0}^{k_{2}}2^{k_{1}}
	\\
	&=
	\frac{1}{4^{m}}
	\sum_{k_{3}=0}^{m}
	\left(\frac{4}{3}\right)^{k_{3}}
	\sum_{k_{2}=0}^{k_{3}}
	\left(\frac{3}{2}\right)^{k_{2}}
	\left(2^{k_{2}+1}-1\right)
	\\
	&=
	\frac{1}{4^{m}}
	\sum_{k_{3}=0}^{m}
	\left(\frac{4}{3}\right)^{k_{3}}
	\left[
	\left(3^{k_{3}+1}-1\right)
	-
	\left(
	\frac{3^{k_{3}+1}}{2^{k_{3}}} - 2
	\right)
	\right]
	\\
	&=
	\frac{1}{4^{m}}
	\sum_{k_{3}=0}^{m}
	\left(\frac{4}{3}\right)^{k_{3}}
	\left(
	3^{k_{3}+1} - 3 \cdot \frac{3^{k_{3}}}{2^{k_{3}}} + 1
	\right)
	\\
	&=
	\frac{1}{4^{m}}
	\sum_{k_{3}=0}^{m}
	\left[
	3 \cdot 4^{k_{3}} - 3 \cdot 2^{k_{3}} + \left(\frac{4}{3}\right)^{k_{3}}
	\right]
	\\
		&=
	4 - 6\left(\frac{1}{2}\right)^{m} + 4\left(\frac{1}{3}\right)^{m} - \left(\frac{1}{4}\right)^{m},
\end{align*}
we get
\begin{eqnarray*}
	Z(s;x,5) &=&-\frac{s(s+1)(s+2)(s+3)(s+4)}{6}\sum\limits_{m=0}^{\infty }%
	\frac{1}{\left( x-1+m\right) ^{s+5}} \\
	&&+\frac{s(s+1)(s+2)(s+3)(s+4)}{4}\sum\limits_{m=0}^{\infty }\frac{1}{%
		2^{m}\left( x-1+m\right) ^{s+5}} \\
	&&-\frac{s(s+1)(s+2)(s+3)(s+4)}{6}\sum\limits_{m=0}^{\infty }\frac{1}{%
		3^{m}\left( x-1+m\right) ^{s+5}} \\
	&&+\frac{s(s+1)(s+2)(s+3)(s+4)}{24}\sum\limits_{m=0}^{\infty }\frac{1}{%
		4^{m}\left( x-1+m\right) ^{s+5}}.
\end{eqnarray*}%
Combining the above equation with definitions of the Hurwitz zeta function,
the Lerch zeta function, we get%
Using the definition of the Hurwitz zeta function together with
Eq.~(\ref{LZ}), we obtain
\begin{align*}
	Z(s;x,5)
	=-&
	\frac{s(s+1)(s+2)(s+3)(s+4)}{24}
	\Bigg\{
	4\zeta(s+5,x-1)
	-6\Phi\left(s+5;\frac{1}{2},x-1\right)
	\\
	&+4\Phi\left(s+5;\frac{1}{3},x-1\right)
	-\Phi\left(s+5;\frac{1}{4},x-1\right)
	\Bigg\}.
\end{align*}
where $x-1\notin \mathbb{Z}_{0}^{-}$ and $Re(s+5)>1$.

We now give the following modification of Eq.  (\ref{DM-1}):
\begin{eqnarray}
	Z(s;x,v) &=&-\frac{s(s+1)\cdots(s+v-1)}{(v-1)!}\sum\limits_{m=0}^{\infty }\frac{1}{(v-1)^{m}}%
	\sum\limits_{k_{v-2}=0}^{m}\left( \frac{v-1}{v-2}\right) ^{k_{v-2}}\cdots
	\sum\limits_{k_{3}=0}^{k_{4}}\left( \frac{4}{3}\right) \notag ^{k_{3}} \\
	&&\times \sum\limits_{k_{2}=0}^{k_{3}}\left( \frac{3}{2}\right)
	^{k_{2}}\sum\limits_{k_{1}=0}^{k_{2}}\frac{2^{k_{1}}}{\left( x-1+m\right) ^{s+v}}. \label{1ii}
\end{eqnarray}%

We give the following open problem involving functional equation for the function  $Z(s;x,v)$: 

We have examined in the special cases above that this function can also be
represented as linear combinations of the functions $\zeta (s,x-1)$ and $%
\Phi \left( s;z,x-1\right) $. That is, its representation in terms of the
functions $\zeta (s,x-1)$ and $\Phi \left( s;z,x-1\right) $ may be possible to
give in the following form:
\[
Z(s;x,v)=-\frac{s(s+1)\cdots(s+v-1)}{(v-1)!}\left((v-1) \zeta
(s+v,x-1)+\sum\limits_{j=2}^{v-1}\alpha _{j}\Phi \left( s+v;\frac{1}{j},x-1\right)
\right),
\]%
where $x-1\notin\mathbb Z_0^{-}$,  $Re(s+v)>1$ and also $\alpha _{j}$ are rational numbers. 

Find the coefficients $\alpha _{j}$?

In the next sections, we give some novel formulas involving relations among the functions $\zeta (s,x-1)$, $%
\Phi \left( s;z,x-1\right) $ and the polynomials $\ell _{n}(x;v)$.

\section{Interpolation function for the polynomials $\ell _{n}(x;v)$} \label{S2}

We show that these infinite series have analytic
continuations on the complex $s$-plane. We also give interpolation function for the polynomials $\ell _{n}(x;v)$ with the aid of integral representation of the function $Z(s;x,v)$ and the Cauchy Residue Theorem.

\subsection{Integral representation of the function $Z(s;x,v)$}

In order to give integral representation of the function $Z(s;x,v)$, we use a similar method to define an interpolation function for the
polynomials $\ell_n(x;v)$. Recall the following integral representation
of the Gamma function:
\[
\Gamma(s+v)
=
y^{s+v}
\int_0^\infty e^{-yw}w^{s+v-1}\,dw,
\]
where
$
\Re(s+v)>0$  and $
\Re(y)>0$ (\textit{cf.} \cite[p. 202]{SrivatavaChoi}). Putting
$
y=m_1+m_2+\cdots+m_{v-1}+x-1,
$
$(m_1,\ldots,m_{v-1}\in\mathbb N_0)$ in the above equation, we obtain
\[
\frac{\Gamma(s+v)}
{\left(m_1+\cdots+m_{v-1}+x-1\right)^{s+v}}
=
\int_0^\infty
e^{-(m_1+\cdots+m_{v-1}+x-1)w}
w^{s+v-1}\,dw.
\]
Multiplying both sides of the previous equation by
$
\frac{1}{\displaystyle\prod_{j=1}^{v-1}j^{m_j}}
$
and summing over $m_1,\ldots,m_{v-1}$ from $0$ to $\infty$, we get
\[
\Gamma(s+v)
\sum_{m_1,\ldots,m_{v-1}=0}^{\infty}
\frac{1}
{\displaystyle
	\left(\prod_{j=1}^{v-1}j^{m_j}\right)
	\left(m_1+\cdots+m_{v-1}+x-1\right)^{s+v}}
\]
\[
=
\int_0^\infty
w^{s+v-1}e^{-(x-1)w}
\prod_{j=1}^{v-1}
\left(
\sum_{m_j=0}^{\infty}
\frac{e^{-m_jw}}{j^{m_j}}
\right)
\,dw,
\]
where $$\sum_{m_1=0}^{\infty}\sum_{m_2=0}^{\infty}\cdots \sum_{m_{v-1}=0}^{\infty}=\sum_{m_1,\ldots,m_{v-1}=0}^{\infty}.$$
Combining this equation with Eq. (\ref{1g2AA}) by replacing $w$ by $-w$, we obtain
\begin{equation}
	\frac{\Gamma(s+v)}{-(v-1)!}\sum_{m_1,\ldots,m_{v-1}=0}^{\infty}
	\frac{1}
	{\displaystyle
		\left(\prod_{j=1}^{v-1}j^{m_j}\right)
		\left(m_1+\cdots+m_{v-1}+x-1\right)^{s+v}}
	=\int_0^\infty
	F(-w,x;v)w^{s-1}\,dw.\label{B}
\end{equation}
In order to give integral representation of the function $Z(s;x,v)$, we not only need the following well-known formula for the Gamma function:
\begin{equation}
	\frac{\Gamma(s+v}{\Gamma(s)}=s(s+1)\ldots(s+v-1), \label{gg1}	
\end{equation} 
	 $v\in \mathbb{N}_0,$ but also define $\Gamma(s)$ for $\Re (s)>-v$ ($v\in \mathbb{N}_0,$) as an analytic function except
for $s=0,-1,-2,\dots -v+1.$ Consequently, the function $\Gamma(s)$ can be continued analytically to the whole
complex $s$-plane except for simple poles at $s \in  \mathbb Z_0^{-}$, for detail see \cite[p. 3]{SrivatavaChoi}). Since $\Gamma(s)$ is a meromorphic function on the whole
complex $s$-plane with simple poles at $s=-v$ ($v\in \mathbb{N}_0,$) with their well-known respective residues are given as follows:
\begin{equation}
	Res(\Gamma(s))_{s=-v}=\frac{(-1)^v}{v!} \label{gg2}
	\end{equation}
for detail see also\cite[p. 4]{SrivatavaChoi}).

Equations (\ref{B}), (\ref{gg1}), (\ref{gg2}) and  (\ref{1g2AA}) motivate the following definition for the function $Z(s;x,v)$:
\begin{equation*}
	Z(s;x,v)
	:=
	\frac{(v-1)!}{\Gamma(s)}
	\int_0^\infty w^{s-1}
	F(-w,x;v)dw.
\end{equation*}
Therefore
\begin{equation}
	Z(s;x,v)=-
	\frac{(v-1)!}{\Gamma(s)}
	\int_0^\infty w^{s+v-1}\frac{e^{-(x-1)w}}{\prod\limits_{j=1}^{v-1}\left( 1-%
	\frac{1}{j}e^{-w}\right) }dw, \label{C}
\end{equation}
 where
$v \in \mathbb{N}$, $
v\geq 2$; $
x-1\in\mathbb C\setminus\mathbb Z_0^-
$; $s \in \mathbb{C}$ and $\Re(s)>v$.

Substituting $v=2$ into \eqref{C}, we arrive at Theorem 2.4 for $n=2$, given by Srivastava and Choi \cite[p. 144]{SrivatavaChoi}. That is
\[
\int_0^\infty w^{s+1}\frac{e^{-(x-1)w}}{1-e^{-w} }dw=-\Gamma(s)\zeta(s+2,x-1),
\]
where
$
\Re(s)>0$,
$x-1\in\mathbb C\setminus\mathbb Z_0^-.$

Putting $v=3$ in \eqref{C}, we have
\begin{equation*}
\int_0^\infty w^{s+2}\frac{e^{-(x-1)w}}{\left( 1-e^{-w}\right) \left( 1-\frac{1}{2}e^{-w}\right)}dw=-\frac{\Gamma(s)}{2}
\sum_{m_1=0}^{\infty}
\Phi
\left(
s+3;
\frac12,
m_1+x-1
\right)
\end{equation*}
	where
	$v \in \mathbb{N}$, $
	v\geq 2$; $
	x-1\in\mathbb C\setminus\mathbb Z_0^-
	$; $s \in \mathbb{C}$ and $\Re(s)>v$.
	Putting $v=4$ in \eqref{C}, we have
	\begin{equation*}
		\int_0^\infty w^{s+3}\frac{e^{-(x-1)w}}{\left( 1-e^{-w}\right) \left( 1-\frac{1}{2}e^{-w}\right)}dw=-\frac{\Gamma(s)}{6}
	\sum_{m_1=0}^{\infty}
	\sum_{m_2=0}^{\infty}
	\frac{1}{2^{m_2}}
	\Phi
	\left(
	s+4;
	\frac13,
	m_1+m_2+x-1
	\right).
\end{equation*}
Combining (\ref{B}) with \eqref{C}, we get
\begin{equation*}
	Z(s;x,v)=-\frac{\Gamma(s+v)}{(v-1)!\Gamma(s)}\sum_{m_1,\ldots,m_{v-1}=0}^{\infty}
\frac{1}
{\displaystyle
	\left(\prod_{j=1}^{v-1}j^{m_j}\right)
	\left(m_1+\cdots+m_{v-1}+x-1\right)^{s+v}}.
\end{equation*}
After some calculations, we get the following a relation between the Lerch zeta function and  the function $Z(s;x,v)$:
\begin{theorem}

\begin{equation}
	Z(s;x,v)=-\frac{\Gamma(s+v)}{(v-1)!\Gamma(s)}\sum_{m_1,\ldots,m_{v-2}=0}^{\infty}
	\frac{1}
	{\displaystyle
		\left(\prod_{j=1}^{v-2}j^{m_j}\right)
		}\Phi
		\left(
		s+v;
		\frac{1}{v-1},
		m_1+\cdots+m_{v-2}+x-1
		\right), \label{ky}
\end{equation}
where 	$v \in \mathbb{N}$, $
v\geq 2$; $
x-1\in\mathbb C\setminus\mathbb Z_0^-
$; $s \in \mathbb{C}$ and $\Re(s)>v$.
\end{theorem}
Substituting $s=-n$ ($n\in \mathbb{N}$) into Eq. (\ref{DM-1}),
and applying Eq. (\ref{ZZ}) with $n$ replaced by $n-v$, we get the following theorem.
\begin{theorem}
	\label{Theo.I}Let $n,v\in \mathbb{N}$ with $v\geq 2$.Then we have%
	\begin{equation}
		Z(-n,x;v)=-\ell _{n}(x;v). \label{ky1}
	\end{equation}
\end{theorem}
\begin{proof}
	In view of equations (\ref{gg1}), (\ref{B}) and (\ref{C}), we define $h(s)$ by means of the following contour integral:
	\begin{equation}
		h(s)=	\int_C w^{s+v-1}\frac{e^{-(x-1)w}}{\prod\limits_{j=1}^{v-1}\left( 1-%
			\frac{1}{j}e^{-w}\right) }dw, \label{cr}
	\end{equation}
	where $C$ denotes the Hankel contour along the cut joining the points $z = 0$ and $z = \infty$ on the real axis,
	which starts from the point at $ \infty$, encircles the origin ($z = 0$) once in the positive (counter-clockwise)
	direction, and returns to the point at $ \infty$ (for details, see also \cite[p.60]{Whittaker} and \cite[p. 204]{SrivatavaChoi}). Here, as usual, we interpret
	zs to mean $\exp (s \log w)$, where we assume log z to be defined by log t on the top part of the real
	axis and by $\log s+2i\pi $ on the bottom part of the real axis, see also \cite{SKS}. Thus the function $	h(s)$ can be written as follows:
	\begin{equation*}
		h(s)= (e^{2i\pi s}-1)\int_{0}^{\infty}t^{s+v-1}\frac{e^{-(x-1)t}}{\prod\limits_{j=1}^{v-1}\left( 1-%
			\frac{1}{j}e^{-t}\right) }dt+\int_{C_{\theta}} w^{s+v-1}\frac{e^{-(x-1)w}}{\prod\limits_{j=1}^{v-1}\left( 1-%
			\frac{1}{j}e^{-w}\right) }dw,
	\end{equation*}
	where $C_{\theta}$denotes the circle of radius $\theta>0$ (and centered at the origin) described in the positive
	(counter-clockwise) direction. Assume first that $\Re(s+v)>1$. It is easy to see that when $\theta \rightarrow 0$, then $%
	\int\limits_{C_{\theta }}\rightarrow 0$, which yields\begin{equation*}
		h(s)= (e^{2i\pi s}-1)\int_{0}^{\infty}t^{s+v-1}\frac{e^{-(x-1)t}}{\prod\limits_{j=1}^{v-1}\left( 1-%
			\frac{1}{j}e^{-t}\right) }dt.
	\end{equation*}
	After some elementary calculations, we have
	\begin{equation*}
		\frac{h(s)}{\Gamma(s)}= (e^{2i\pi s}-1)Z(s;x,v).
	\end{equation*}
	By the principle of analytic continuation, the above equation can be shown to hold for all $s \notin \mathbb{Z}$. The above computations
	provide us with an analytic continuation of the function $Z(s;x,v)$. When $s \rightarrow -n$ ($n \in \mathbb{N}$), with the aid of (\ref{gg2}), we obtain $$\Gamma(s)(e^{2i\pi s}-1) \rightarrow \frac{(-1)^n 2i\pi}{n!}$$ and 
	\begin{equation*}
		h(-n)= 2i\pi Res_{s=-n}\left\{w^{s-1}F(-w,x;v\right\}=-2i\pi \frac{(-1)^n}{n!}\ell _{n}(x;v).
	\end{equation*}
	By applying the Cauchy Residue
	Theorem to Eq. (\ref{cr}), proof of theorem is completed by suitably combining the above equations.
\end{proof}
We note that due to Theorem \ref{Theo.I}, the function $%
Z(s,x;v)$ is an interpolation function of the polynomials $\ell _{n}(x;v)$.

We now give other construction of the function $%
Z(s,x;v)$. By using (\ref{dp1}), we have%
\begin{equation}
	\frac{w^{v}e^{(x-1)w}}{(-1)^{v-1}(v-1)!\prod\limits_{j=1}^{v-1}\left( 1-%
		\frac{1}{j}e^{w}\right) }=\sum_{n=0}^{\infty }\ell _{n}(x;v)\frac{w^{n}}{n!}.
	\label{gs1}
\end{equation}%
Assuming that $Re(w)<0$, so $\left\vert \frac{1}{j}e^{w}\right\vert <1$, ($1\leq j\leq v-1$). Thus Eq. (\ref{gs1}) reduces to%
\begin{eqnarray*}
	\sum_{n=0}^{\infty }\ell _{n}(x;v)\frac{w^{n}}{n!} &=&\frac{w^{v}e^{w(x-1)}}{%
		(-1)^{v-1}(v-1)!}\prod\limits_{j=1}^{v-1}\sum_{m=0}^{\infty }\frac{1}{j^{m}}%
	e^{wm} \\
	&=&\frac{w^{v}e^{w(x-1)}}{(-1)^{v-1}(v-1)!}\sum_{m=0}^{\infty
	}\frac{1}{1^m}e^{wm}\sum_{m=0}^{\infty }\frac{1}{2^{m}}e^{wm}
	\dots\sum_{m=0}^{\infty }\frac{1}{(v-1)^{m}}e^{wm}.
\end{eqnarray*}%
Therefore
\begin{eqnarray*}
	\sum_{n=0}^{\infty }\ell _{n}(x;v)\frac{w^{n}}{n!} &=&\frac{w^{v}}{%
		(-1)^{v-1}(v-1)!}\sum\limits_{m=0}^{\infty } \sum\limits_{k_{v-2}=0}^{m}\\
	&&\times \cdots
	\sum\limits_{k_{3}=0}^{k_{4}}\sum\limits_{k_{2}=0}^{k_{3}}\sum%
	\limits_{k_{1}=0}^{k_{2}}\frac{e^{w(x-1+m)}}{1^{k_{1}}2^{k_{2}-k_{1}}3^{k_{3}-k_{2}}4^{k_{4}-k_{3}}%
		\cdots (v-1)^{m-k_{v-2}}},
\end{eqnarray*}%
which yields
\begin{eqnarray*}
	\sum_{n=0}^{\infty }\ell _{n}(x;v)\frac{w^{n}}{n!} &=&\frac{1}{%
		(-1)^{v-1}(v-1)!}\sum\limits_{m=0}^{\infty }\sum\limits_{n=0}^{\infty
	}\left(x-1+m\right) ^{n}\frac{w^{n+v}}{n!} \\
	&&\times \sum\limits_{k_{v-2}=0}^{m}\cdots
	\sum\limits_{k_{3}=0}^{k_{4}}\sum\limits_{k_{2}=0}^{k_{3}}\sum%
	\limits_{k_{1}=0}^{k_{2}}\frac{1}{1^{k_{1}}2^{k_{2}-k_{1}}3^{k_{3}-k_{2}}4^{k_{4}-k_{3}}%
		\cdots (v-1)^{m-k_{v-2}}}.
\end{eqnarray*}%
Thus
\begin{eqnarray*}
	\sum_{n=0}^{\infty }\ell _{n}(x;v)\frac{w^{n}}{n!} &=&\frac{1}{%
		(-1)^{v-1}(v-1)!}\sum\limits_{m=0}^{\infty }\sum\limits_{n=0}^{\infty }%
	\binom{n}{v}v!\left( x-1+m\right) ^{n-v}\frac{w^{n}}{n!} \\
	&&\times \sum\limits_{k_{v-2}=0}^{m}\cdots
	\sum\limits_{k_{3}=0}^{k_{4}}\sum\limits_{k_{2}=0}^{k_{3}}\sum%
	\limits_{k_{1}=0}^{k_{2}}\frac{1}{1^{k_{1}}2^{k_{2}-k_{1}}3^{k_{3}-k_{2}}4^{k_{4}-k_{3}}%
		\cdots (v-1)^{m-k_{v-2}}}.
\end{eqnarray*}%
After equating the coefficients of the term $\frac{w^{n}}{n!}$ on both sides
of the above equation, we get the following theorem:

\begin{theorem}
	Let $n,v\in \mathbb{N}$ be with $v\geq 3$. Then%
	\begin{eqnarray}
		&&\sum\limits_{m=0}^{\infty }\sum\limits_{k_{v-2}=0}^{m}\cdots
		\sum\limits_{k_{3}=0}^{k_{4}}\sum\limits_{k_{2}=0}^{k_{3}}\sum%
		\limits_{k_{1}=0}^{k_{2}}\frac{\left( x-1+m\right) ^{n-v}}{
			1^{k_{1}}2^{k_{2}-k_{1}}3^{k_{3}-k_{2}}4^{k_{4}-k_{3}}%
			\cdots (v-1)^{m-k_{v-2}}}  \label{ZZ}
		\\
		&=&(-1)^{v-1}\frac{(v-1)!}{n(n-1)\ldots(n-v+1)}\ell _{n}(x;v).  \notag
	\end{eqnarray}
	In particular, for $v=2$, we get
	\begin{eqnarray}
		\sum\limits_{m=0}^{\infty }\left( x-1+m\right) ^{n-2}
		=-\frac{\ell _{n}(x;2)}{n(n-1)}.  \label{ZZ-I}
	\end{eqnarray}
\end{theorem}
\begin{remark}
Substituting $n=-s$ into Eq. \textup{(\ref{ZZ})} and using Eq.(\ref{ky1}) and (\ref{ky}), we also derive
definition of the function $Z(s,x;v)$, given in (\ref{DM-1}).
\end{remark}
\begin{remark}
	Observe that  Eq. \textup{(\ref{ZZ})} gives us the values of the
	function $Z(s,x;v)$ at negative integers. Since 
	\begin{eqnarray*}
		\zeta (-n,x-1)=	\sum\limits_{m=0}^{\infty }\left( x-1+m\right) ^{n}
		=-\frac{B_{n+1}(x-1)}{n+1},  
	\end{eqnarray*}
	Eq. \textup{(\ref{ZZ-I})} reduces to the following result:
	\begin{eqnarray*}
		\ell _{n+2}(x;2)	=(n+2)B_{n+1}(x-1) 
	\end{eqnarray*}
	(\textit{cf.} \textup{\cite[Remark 2.4]{KilarMontes}}) and also for $x=0$ in the above equation, we get
	\begin{eqnarray*}
		\ell _{n+2}(2)	=(n+2)B_{n+1}(-1) 
	\end{eqnarray*}
	(\textit{cf.} \textup{\cite[Corollary 5]{SimsekMMAS2021}}).
\end{remark}
Using Eq. (\ref{ZZ}), we define the function (\ref{DM-1}), which is not only related to the function $\Phi (z,s,x)$, but also interpolates the polynomials $\ell _{m}(x;v)$ at negative integers.
The main purpose of this section is to provide values of the series $Z(s,x;v)$ at negative integers. It is shown that this series definitely coincides with
Eq. (\ref{ZZ}).

\section{Applying Laplace transform to equation (\protect\ref{dp7})} \label{S3}

In this section, we apply the Laplace transform to equation (\ref{dp7}); we
find infinite series representation for the polynomials $\ell _{n}(x;v)$.

Using Eq. (\ref{dp7}), we get%
\begin{equation*}
	e^{-(1-\frac{x}{2})w}\sum_{n=0}^{\infty }\ell _{n}(x;v)\frac{(-w)^{n}}{n!}=(-w)^{v}	e^{-\frac{x}{2}w}\prod\limits_{j=1}^{v-1}\sum_{n_j=0}^{\infty }\frac{1}{j^{n_j}}e^{-n_jw}
\end{equation*}%
By applying the Laplace transform to the above equation with the aid of the
uniform convergence property of series, we obtain%
\begin{equation*}
	\sum_{n=0}^{\infty }(-1)^n\ell _{n}(x;v)\frac{1}{n!}\int\limits_{0}^{\infty
	}w^{n}e^{-(1-\frac{x}{2})}dw=(-1)^v\sum_{m_1,\ldots,m_{v-1}=0}^{\infty}\prod\limits_{j=1}^{v-1}\frac{1}{j^{n_j}}\int\limits_{0}^{\infty }w^{v}e^{-(m_1+m_2+\ldots +m_{v-1}+\frac{x}{2})w}dw,
\end{equation*}%
where we assume that $1-\frac{x}{2}>0$ and $\frac{x}{2}>0;$  $\sum_{m_1=0}^{\infty}\sum_{m_2=0}^{\infty}\cdots \sum_{m_{v-1}=0}^{\infty}=\sum_{m_1,\ldots,m_{v-1}=0}^{\infty}.$ Thus, we get%
\begin{equation*}
	\sum_{n=0}^{\infty }\frac{(-1)^n\ell _{n}(x;v)}{n!(1-\frac{x}{2})^{n+1}}
	\int\limits_{0}^{\infty
	}u^{n}e^{-u}du=(-1)^v\sum_{m_1,\ldots,m_{v-1}=0}^{\infty}\frac{\prod\limits_{j=1}^{v-1}\frac{1}{j^{m_j}}}{(m_1+m_2+\ldots +m_{v-1}+\frac{x}{2})^{v+1}}\int\limits_{0}^{\infty }y^{v}e^{-y}dy.
\end{equation*}%
Since%
\begin{equation*}
	\int\limits_{0}^{\infty }y^{v}e^{-y}dy=\Gamma (v+1)=v!,
\end{equation*}%
we obtain%
 the
following theorem:

\begin{theorem}
	Let $v\in \mathbb{N}$ be with $v\geq 2$. Assuming that $\left\vert 1-\frac{x}{2}\right\vert
	>1$. Then%
	\begin{equation}
		\sum_{m_1,\ldots,m_{v-1}=0}^{\infty}\frac{\prod\limits_{j=1}^{v-1}\frac{1}{j^{m_j}}}{(m_1+m_2+\ldots +m_{v-1}+\frac{x}{2})^{v+1}}=\frac{1}{v!}\sum_{n=0}^{\infty }(-1)^{n+v}\frac{\ell _{n}(x;v)}{\left(1-\frac{x}{2}\right)^{n+1}},  \label{PD-4}
\end{equation}
where $\frac{x}{2} \notin\mathbb Z_0^{-}$.
\end{theorem}

Note that there are many applications of Eq. (\ref{PD-4}). To
compute the values of the zeta functions involving the Riemann zeta
function, the Hurwitz zeta function, the Lerch zeta function, etc, there are
many different methods and techniques. We now give. For instance, putting $%
v=2$ in Eq. (\ref{PD-4}), we get the following result:

\begin{corollary}
	Assuming that $\left\vert 1-\frac{x}{2}\right\vert
	>1$. Then
	\begin{equation*}
		\sum_{n=0}^{\infty }(-1)^{n}\frac{\ell _{n}(x;2)}{\left(1-\frac{x}{2}\right)^{n+1}}=2\zeta\left(3,\frac{x}{2}\right),
	\end{equation*}
	where $\frac{x}{2} \notin\mathbb Z_0^{-}$.
\end{corollary}
Therefore, we show that Eq. (\ref{PD-4}) can be given in terms of the Hurwitz Lerch zeta function by the following theorem:

\begin{theorem}
	Let $v\in \mathbb{N}$ be with $v\geq 3$. Assuming that $\left\vert 1-\frac{x}{2}\right\vert
	>1$. Then%
	\begin{equation}
		\sum_{m_1,\ldots,m_{v-2}=0}^{\infty}\prod\limits_{j=1}^{v-2}\frac{1}{j^{m_j}}\Phi \left(v+1,\frac{1}{v-1},m_1+m_2+\ldots +m_{v-2}+\frac{x}{2}\right)=\sum_{n=0}^{\infty }\frac{(-1)^{n+v}v!\ell _{n}(x;v)}{\left(1-\frac{x}{2}\right)^{n+1}},
		 \label{PD-4x}
	\end{equation}
	where $\frac{x}{2} \notin\mathbb Z_0^{-}$.
\end{theorem}
\section*{Funding}
There is no funding for this manuscript

\section*{Data availability}
No data were used in the conduct of this manuscript.

\section*{Declarations}
\textbf{Conflict of interest} No potential Conflict of interest was reported by the authors.

\end{document}